 \documentclass[12pt]{amsart}
 \usepackage{amssymb}
 \usepackage{amsmath,amsxtra}
\usepackage{amsfonts}

\numberwithin{equation}{section}

 \newtheorem{thm}{Theorem}[section]
 \newtheorem{lemma}[thm]{Lemma}
 \newtheorem{prop}[thm]{Proposition}
 \newtheorem{cor}[thm]{Corollary}

 \theoremstyle{definition}
\newtheorem{defn}[thm]{Definition}
\newtheorem*{example}{Example}
\newtheorem{exa}{Example}

\theoremstyle{definition}
\newtheorem{problem}{Open Problem}

\theoremstyle{remark}
\newtheorem{remark}[thm]{Remark}

\newcommand{ \Mf}{\mathfrak{M}}
\newcommand{\e}{\mathfrak{e}}
\newcommand{\Lc}{\mathcal{L}}

\newcommand{\Sc}{\mathcal{S}}
\newcommand{\Gb}{\mathbb G}
\newcommand{ \Rb}{\mathbb{R}}

 \newcommand{ \Nbb}{\mathbb{N}}

 \newcommand{\al}{\alpha}
 \newcommand{\la}{\langle}
 \newcommand{\ra}{\rangle}
\newcommand{\de}{\delta}
 \newcommand{\be}{\beta}
 
 \newcommand{\ep}{\varepsilon}

\begin{document}

 \title[Algebraic and analytic properties of semigroups]{Algebraic and analytic properties of semigroups related to fixed point properties of non-expansive mappings}

\author[A. T.-M. Lau]{Anthony To-Ming Lau \dag}
\address{\dag \; Department of Mathematical and Statistical sciences\\
           University of Alberta\\
           Edmonton, Alberta\\
           T6G 2G1 Canada}
\email{anthonyt@ualberta.ca}
\thanks{\dag  \; Supported by NSERC Grant ZC912}

\author[Y. Zhang]{ Yong Zhang \ddag}
\address{\ddag \; Department of Mathematics\\
           University of Manitoba\\
           Winnipeg, Manitoba\\
           R3T 2N2 Canada}
\email{yong.zhang@umanitoba.ca}
\thanks{\ddag \; Supported by NSERC Grant 1280813}


\subjclass{Primary 43A07 Secondary 43A60, 22D05 46B20}

\keywords{amenability properties, semigroups, nonexpansive mappings, weakly compact, weak* compact, convex sets,
                                    common fixed point, invariant mean, submean}

\begin{abstract}
The purpose of this paper is to give an updated survey on various algebraic and analytic properties of semigroups related to fixed point properties of semigroup actions on a non-empty closed convex subset of a Banach space or, more generally, a locally convex topological vector space.
\end{abstract}

\maketitle

Dedicated to Professor Karl H. Hofmann with admiration and \mbox{respect}

\section{Introduction}\label{intro}

In this paper we shall present an updated survey on fixed point properties of a semigroup of non-expansive mappings on a closed convex subset of a Banach space (or more generally, a locally convex topological space) related to certain algebraic and analytic properties of the semigroup. We begin with some historic developments of the investigation.

Let $E$ be a Banach space and let $K$ be a non-empty bounded closed convex subset of $E$. We say that $K$ has the \emph{fixed point property} if for every non-expansive mapping $T$: $K\to K$ (i.e. $\|Tx-Ty\| \leq \|x-y\|$, $x, y\in K$), $K$ contains a fixed point for $T$.

It follows from Bruck \cite{Bruck} that if $E$ is a Banach space with the \emph{weak fixed point property} (i.e. any weakly compact convex subset of $E$ has the fixed point property), then any weakly compact convex subset $K$ of $E$ has the (common) fixed point property for any commutative semigroup acting on $K$.

A well-known result of Browder \cite{Browder} asserts that if $E$ is uniformly convex, then $E$ has the weak fixed point property. Kirk \cite{Kirk} extended this result by showing that if $K$ is a weakly compact subset of $E$ with normal structure, then $K$ has the fixed point property. Other examples of Banach spaces with the weak fixed point property include $c_0$, $\ell^1$, trace class operators on a Hilbert space and the Fourier algebra of a compact group (see \cite{DLT, G-K2, L-L, LM, LMU, Lennard, Lim 80, Maurey, Brailey} and \cite{BP, BPP} for more details). However, as shown by Alspach \cite{A}, $L^1[0,1]$ does not have the weak fixed point property.

Given a semigroup $S$, we let $\ell^{\infty}(S)$ be the C*-algebra of bounded complex-valued functions on $S$ with the supremum norm and pointwise multiplication. For each $a\in S$ and $f\in \ell^{\infty}(S)$ let $\ell_a f$ and $r_a f$ be the left and right translates of $f$ by $a$, respectively; i.e. $\ell_af(s) = f(as)$ and $r_af(s) = f(sa)$ ($s\in S$). Let $X$ be a closed subspace of  $\ell^{\infty}(S)$ containing constants and invariant under translations. Then a linear functional $m\in X^*$ is called a \emph{mean} if $\|m\| = m(1) = 1$; $m$ is called a \emph{left (resp. right) invariant mean}, denoted by LIM (resp. RIM), if $m(\ell_af) = m(f)$ (resp. $m(r_af) = m(f)$) for all $a\in S$, $f\in X$. A discrete semigroup $S$ is \emph{left (resp. right) amenable} if $\ell^{\infty}(S)$ has a LIM (resp. RIM). Let X be a C*-subalgebra of $\ell^\infty(S)$. Then the \emph{spectrum} of X is the set of non-zero multiplicative linear functionals on X equipped with the relative weak* topology. Every $s\in S$ is a multiplicative linear functional on X if we regard it as the evaluation functional: $\la s, f\ra = f(s)$ for $f\in X$.

Let $S$ be a \emph{semitopological semigroup}, i.e.~$S$ is a semigroup with Hausdorff topology such that for each $a\in S$, the mappings $s \mapsto sa$ and $s \mapsto as$ from $S$ into $S$ are continuous. We let $C_b(S)$ be the space of all bounded continuous complex-valued functions on $S$. We denote by $LUC(S)$ the space of all left uniformly continuous functions on $S$, i.e. all $f\in C_b(S)$ such that the mapping $s\mapsto \ell_s f$: $S\to C_b(S)$ is continuous. The semitopological semigroup $S$ is called left amenable if $LUC(S)$ has a LIM. We note that if $S$ is discrete, then $LUC(S) = C_b(S) = \ell^\infty(S)$ and so left amenability of $S$ coincides with that previously defined. Now
 denote by $AP(S)$ the space of all $f\in C_b(S)$ such that $\mathcal{L}\mathcal{O}(f) = \{\ell_sf: \; s\in S\}$ is relatively compact in the norm topology of $C_b(S)$, and denote by $WAP(S)$ the space of all $f\in C_b(S)$ such that $\mathcal{L O}(f)$ is relatively compact in the weak topology of $C_b(S)$. Functions in $AP(S)$ (resp. $WAP(S)$) are called almost periodic (resp. weakly almost periodic) functions on $S$. Later in this paper we will also need to consider the set $\mathcal{R}\mathcal{O}(f) = \{r_sf: \; s\in S\}$. As well known, $f\in AP(S)$ (resp. $f\in WAP(S)$) if and only if $\mathcal{R}\mathcal{O}(f)$ is relatively compact in the norm (resp. weak) topology of  $C_b(S)$. Let $S^a$ (resp. $S^w$) be the almost periodic (resp. weakly almost periodic) compactification of $S$, i.e. $S^a$ (resp $S^w$) is the spectrum of the C*-algebra $AP(S)$ (resp. $WAP(S)$). Then $S^a$ and $S^w$ are semitopological semigroups with multiplications defined by: $\la m\cdot n, f\ra = \la m,n\cdot f\ra$, where $n\cdot f(s) = \la n, \ell_sf \ra$, $m, n \in S^a$ (resp. $S^w$), $f \in AP(G)$ (resp. $WAP(G)$).  In fact, the multiplication in $S^a$ is even jointly continuous. In other words, $S^a$ is a topological semigroup.

The semitopological semigroup $S$ is called \emph{left reversible} if any two closed right ideals of $S$ have non-void intersection, i.e. $\overline{aS}\cap\overline{bS} \neq \emptyset$ for any $a, b\in S$. 

If $A$ is a subset of a topological space $E$, then $\overline{A}$ will denote the closure of $A$ in $E$. If in addition, $E$ is a linear topological space, then [$\overline{co} A$] $co A$ will denote the [closed] convex hull of $A$ in $E$. 

An action of $S$ on a topological space $K$ is a mapping $\psi$ from $S\times K$ into $K$, denoted by $T_sx = \psi(s,x)$ ($s\in S$ and $x\in K$), such that $T_{s_1s_2}x = T_{s_1}(T_{s_2}x)$ ($s_1, s_2\in S$ and $x\in K$). The action is \emph{separately continuous} or \emph{jointly continuous} if the mapping $\psi$ is, respectively, separately or jointly continuous. For convenience, $T_sx$ is also denoted by $sx$, and we call $\Sc = \{T_s: s\in S\}$ a representation of $S$ on $K$. 

When $K$ is a convex subset of a linear topological space, we say that an action of $S$ on $K$ is affine if for each $s\in S$, the mapping  $x \mapsto sx$: $K\to K$ satisfies $s(\lambda x + (1- \lambda)y) = \lambda sx + (1-\lambda)sy$ for $s\in S$, $x, y\in K$ and $0\leq \lambda \leq 1$. A locally convex topological space $E$ with the topology generated by a family $Q$ of seminorms will be denoted by $(E,Q)$.
An action of $S$ on a subset $K\subseteq E$ is $Q$-non-expansive if $\rho (s\cdot x - s\cdot y) \leq \rho(x-y)$ for all $s\in S$, $x, y \in K$ and $\rho \in Q$.

The following fixed point property was proved by the first author \cite[Theorem~4.1]{Lau73} (see also \cite{M1, Tak}):
\begin{thm}
Let $S$ be a semitopological semigroup. Then $AP(S)$, the space of continuous almost periodic functions on $S$, has a LIM if and only if $S$ has the following fixed point property:
\begin{description}
\item[(D)] Whenever $S$ is a separately continuous and $Q$-non-expansive action on a compact convex subset $K$ of a separated locally convex space $(E, Q)$, $K$ has a common fixed point for $S$.
\end{description}
\end{thm}

It has been an open question for quite a long time (see \cite{Lau76, Lau.survey}) as to whether the existence of LIM on $WAP(S)$, the space of continuous weakly almost periodic functions on $S$, can be characterized by a fixed point property for non-expansive actions of $S$ on a weakly compact convex set.

It was proved by Hsu \cite{Hsu} (also see \cite[Corollary 5.5]{L-T1}) that if $S$ is discrete and left reversible, then $S$ has the following fixed point property:
\begin{description}
\item[(G)] Whenever $S$ acts on a weakly compact convex subset $K$ of a separated locally compact convex space $(E, Q)$ and the action is weakly separately continuous and $Q$-non-expansive, then $K$ contains a common fixed point for $S$.
\end{description}
Since the fixed point property (G) implies that $WAP(S)$ has LIM, it follows that if $S$ is discrete and left reversible, then $WAP(S)$ has a LIM. This improved an earlier result of Ryll-Nardzewski who proved, using his fixed point theorem for affine maps on weakly compact convex subsets of a Banach space, the existence of LIM on $WAP(S)$ when S is a group (see \cite{greenleaf}).

It is known that if $S$ is discrete and left amenable, then $S$ is left reversible. However a general semitopological semigroup $S$ needs not be left reversible even when $C_b(S)$ has a LIM unless $S$ is normal (see \cite{HL}).

The following implication diagram summarizes the known relations mentioned above for a discrete semigroup $S$.
\begin{equation}\label{E discrete}
\begin{matrix}
S \text{ left amenable} & & \\
     \Downarrow \quad \mathrel{\not\Uparrow} & & \\
S \text{ left reversible} & & \\
                   \negthickspace                                 &\searrow & \\
     \Downarrow \quad \mathrel{\not\Uparrow}  & & WAP(S) \text{ has LIM}\\
                    \negthickspace                                &\swarrow & \\
AP(S)\text{ has LIM}         &                  &
\end{matrix}
\end{equation}

The implication``$S$ is left reversible $\Rightarrow$ $AP(S)$ has a LIM'' for any semitopological semigroup was established in \cite{Lau73}. During the 1984 Richmond, Virginia conference on analysis on semigroups, T. Mitchell 
gave two examples to show that, for a discrete semigroup $S$, ``$AP(S)$ has LIM'' $\nRightarrow$ ``$S$ is left reversible'' (see \cite{Lau90}). 

\section{Semigroups of non-expansive mappings}\label{nonexpansive}

An action of a semitopological semigroup $S$ on a Hausdorff space $X$ is called \emph{quasi-equicontinuous} if $\overline{S}^{\,p}$, the closure of $S$ in the product space $X^X$, consists of only continuous mappings. Obviously, an equicontinuous action on a closed subset of a topological vector space is always quasi-equicontinuous (simply because if a net of equicontinuous functions converges pointwise to a function, then the limit function is also continuous). But a quasi-equicontinuous action on a convex compact subset of a topological vector space may not be equicontinuous (\cite[Example 4.14]{L-Z1}). The following was proved in \cite{L-Z1}.

\begin{thm}\label{F}
Let $S$ be a separable semitopological semigroup. Then $WAP(S)$ has a LIM if and only if $S$ has the fixed point property (F) stated as follows.
\begin{description}
\item[(F)] Whenever $S$ acts on a weakly compact convex subset $K$ of a separated locally convex space $(E, Q)$  and the action is weakly separately continuous, weakly quasi-equicontinuous and $Q$-nonexpansive, then $K$ contains a common fixed point for $S$.
\end{description}
\end{thm}

Consider the following fixed point property for a semitopological semigroup $S$:
\begin{description}
\item[(E)]\label{E} Whenever $S$ acts on a weakly compact convex subset $K$ of a separated locally convex space $(E, Q)$ as $Q$-nonexpansive self-mappings and, if in addition, the action is separately continuous and equicontinuous when $K$ is equipped with the weak topology of  $(E. Q)$, then $K$ contains a common fixed point for $S$.
\end{description}
Clearly we have
\[  (\text{G}) \Rightarrow  (\text{F}) \Rightarrow  (\text{E}) \Rightarrow  (\text{D}). \]

\begin{problem}
 Can any of the above implications be reversed?
\end{problem}

Related to multiplicative means, the following result was obtained in \cite{L-Z1}.

\begin{thm}\label{Multiplicative}
Let $S$ be separable and $n$ be a positive integer. Then $WAP(S)$ has a LIM of the form $\frac{1}{n}\sum_{i=1}^n \phi_i$, where each $\phi_i$ is a multiplicative mean on $WAP(S)$, if and only if
\begin{description}
\item[($P_n$)] Whenever $S$ is a separately continuous and quasi-equicontinuous action on a compact Hausdorff space $X$, then there exists a nonempty finite subset $F\subseteq X$, $|F| \leq n$, $|F|$ divides $n$ such that $sF = F$ for all $s \in S$.
\end{description}
\end{thm}

Let ($E,Q$) be a separated locally convex space. A subset $K$ of $E$ is said to have $Q$-normal structure 
if, for each $Q$-bounded  subset $H$ of $K$ that contains more than one point, there is $x_0\in co H$ and $p\in Q$ such that $\sup\{p(x-x_0) :\; x\in H\} < \sup\{p(x-y) :\; x, y\in H\}$. Here by $Q$-boundedness of $H$ we mean for each $p \in Q$ there is $d>0$ such that $p(x) \leq d$ for all $x\in H$. Any $Q$-compact subset has $Q$-normal structure. In a uniformly convex space (eg. any $L^p$, $p>1$, space) a bounded convex set always has normal structure.

For $AP(S)$ to have a LIM the following two results were proved in \cite{L-Z1}.

\begin{thm}\label{AP E'}
Let $S$ be a semitopological semigroup. Then $AP(S)$ has LIM if and only if $S$ has the following fixed point property.
\begin{description}
\item[(E$'$)]\label{E'} Whenever $S$ acts on a weakly compact convex subset $K$ of a separated locally convex space ($E, Q$) as $Q$-nonexpansive mappings, if $K$ has $Q$-normal structure and the $S$-action is separately continuous and equicontinuous when $K$ is equipped with the weak topology of $(E, Q)$, then $K$ contains a common fixed point for $S$.
\end{description}
In particular, fixed point properties (D) and (E\,$'$) are equivalent.
\end{thm}

\begin{thm}\label{AP for separable S} Let $S$ be a separable semitopological semigroup. Then $AP(S)$ has a LIM if and only if the fixed point property (E) holds.
\end{thm}

We wonder whether one can remove the separability condition on $S$ in Theorems~\ref{F}, \ref{Multiplicative} and \ref{AP for separable S}.
When $n= 1$ we can answer this question for Theorem~\ref{Multiplicative} affirmatively (see \cite[Theorem~5.3]{L-Z2}).
\begin{thm}\label{WAP MLIM}
$WAP(S)$ has a multiplicative LIM if and only if whenever $S$ is a separately continuous and quasi-equicontinuous action on a compact Hausdorff space $X$, then $X$ has a common fixed point for $S$.
\end{thm}

The bicyclic semigroup is the semigroup generated by a unit $e$ and two more elements $p$ and $q$ subject to the relation $pq = e$. We denote it by $S_1 = \la e, p, q \;| \;pq = e\ra$. The semigroup generated by a unit $e$ and three more elements $a, b$ and $c$ subject to the relations $ab = ac = e$ is denoted by $S_2 = \la e, a, b, c \;|\; ab = e, ac = e\ra$; and the semigroup generated by a unit $e$ and four more elements $a, b, c,d$ subject to the relations $ac = bd = e$ is denoted by $S_{1,1} = \la e, a, b, c,d \;| \;ac = e, bd = e\ra$. We call $S_2$ and $S_{1,1}$ partially bicyclic semigroups. Duncan and Namioka showed in \cite{D-N} that $S_1$ is an amenable semigroup by revealing the maximal group homomorphic image of $S_1$. The authors gave a direct proof to the result in \cite{L-Z1} by constructing a left and a right invariant mean on $\ell^{\infty}(S_1)$. The partially bicyclic semigroups $S_2$ and $S_{1,1}$ were also investigated in \cite{L-Z1}. The results can be summarized as follows.

\begin{thm}\label{bicyclic} 
Regarding the semigroups $S_1$, $S_2$ and $S_{1,1}$, the following properties hold.
\begin{enumerate}
\item The bicyclic semigroup $S_1$ is amenable;
\item The partially bicyclic semigroup $S_{1,1}$ is neither left nor right amenable, while the partially bicyclic semigroup $S_2$ is right amenable but not left amenable;
\item Both $AP(S_{1,1})$ and $AP(S_2)$ have an invariant mean;
\item $WAP(S_2)$ has a LIM, while $WAP(S_{1,1})$ has no LIM.
\end{enumerate}
\end{thm}

We note that both $S_2$ and $S_{1,1}$ are not left reversible. The above theorem shows that there are semigroups $S$ such that $AP(S)$ has a LIM but $WAP(S)$ has no LIM, and that there are semigroups $S$ such that $WAP(S)$ has a LIM but $S$ is not left reversible. So the theorem resolved problem 27 raised in \cite{Lau90} and problem 1 raised in \cite{Lau76}. As a consequence, for a discrete semigroup $S$, the diagram (\ref{E discrete}) may be completed to:
\[
\text{left amenable}\,
\begin{smallmatrix}
\Rightarrow\\
\nLeftarrow
\end{smallmatrix}\,
 \text{left reversible}\,
\begin{smallmatrix}
\Rightarrow\\
\nLeftarrow
\end{smallmatrix}\,
WAP(S)\, \text{has LIM}\,
\begin{smallmatrix}
\Rightarrow\\
\nLeftarrow
\end{smallmatrix}\,
AP(S)\, \text{has LIM}
\]

Consider fixed point properties (F*) and (G*), which are, respectively, the fixed point properties (F) and (G) with the words ``separately continuous'' replaced by ``jointly continuous''. Clearly,  (G*) implies (F*). The next two results were obtained in \cite{L-Z1}.

\begin{thm}
Let $S$ be a separable semitopological semigroup. Then fixed point property (F*) holds if and only if the space $WAP(S)\cap LUC(S)$ has a LIM.
\end{thm}

\begin{thm}
Let $S$ be a left reversible and metrizable semitopological semigroup. Then $S$ has the fixed point property (G*). In particular, $WAP(S)\cap LUC(S)$ has a LIM.
\end{thm}

The following diagram summarizes the relations among the fixed point properties discussed in \cite{L-Z1}.

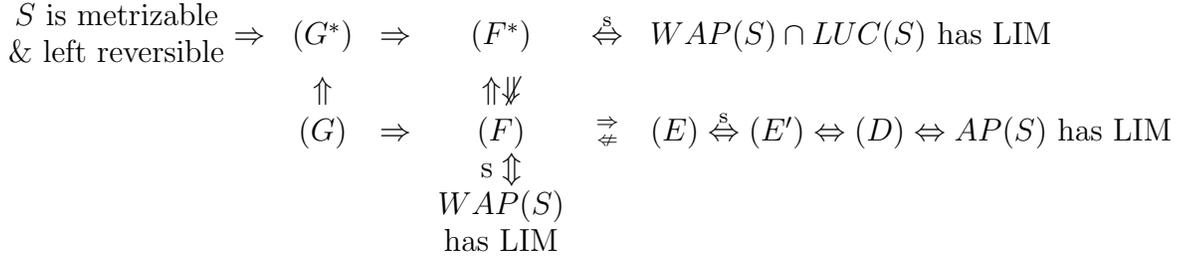
\begin{equation*}
\begin{array}{cccccl}
\begin{matrix}S \text{ is metrizable}\\ \text{\& left reversible}\end{matrix} \Rightarrow &(G^*) &\Rightarrow&(F^*)& \overset{\text{s}}{\Leftrightarrow}& WAP(S)\cap LUC(S) \text{ has LIM}\\
     &\Uparrow& &\Uparrow \mathrel{\not\Downarrow} && \\
&(G)& \Rightarrow&(F)& \begin{smallmatrix}\Rightarrow\\ \nLeftarrow\end{smallmatrix} &(E) \overset{\text{s}}{\Leftrightarrow} (E') \Leftrightarrow (D) \Leftrightarrow AP(S)\text{ has LIM}\negthickspace\\
                   &&&\text{s}\Updownarrow                                & \\
&&& \begin{matrix}WAP(S)\\ \text{has LIM}\end{matrix}                        &   &
\end{array}
\end{equation*}
where ``s'' means the implication is under the condition that the semigroup is separable.

Let $S$ be a semitopological semigroup. The set MM$(S)$ of all multiplicative means on $C_b(S)$ is a total subset of $C_b(S)^*$. It generates a locally convex topology $\tau_{MM} = \sigma(C_b(S), \text{MM}(S))$ on $C_b(S)$. A function $f\in C_b(S)$ is called \emph{left multiplicatively continuous} if the mapping $s\mapsto \ell_sf$ from $S$ into ($C_b(S),\tau_{MM})$ is continuous \cite{Mitch_LUC}.  Let $LMC(S)$ be the subspace of $C_b(S)$ consisting of all left multiplicatively continuous functions on $S$. Then $LMC(S)$ is a left invariant closed subalgebra of $C_b(S)$ containing the constant function. The following result is due to  T. Mitchell \cite{Mitch_LUC}. An alternative proof was given in \cite{L-Z2}.

\begin{thm}\label{Thm Sr}
Let $S$ be a semitopological semigroup. Then
\begin{enumerate}
\item $LUC(S)$ has a multiplicative left invariant mean if and only if 
\begin{description}
\item[($F_j$)] every jointly continuous representation of $S$ on a nonempty compact Hausdorff space has a common fixed point for $S$;
\end{description}
\item $LMC(S)$ has a multiplicative left invariant mean if and only if
\begin{description}
\item[($F_{s}$)] every separately continuous representation of $S$ on a nonempty compact Hausdorff space has a common fixed point for $S$.
\end{description}
\end{enumerate}
\end{thm}

 Recall further that a function $f\in C_b(S)$ is a \emph{weakly left uniformly continuous} if the mapping $s\mapsto \ell_sf$ from $S$ into ($C_b(S),\text{wk})$ is continuous \cite{Mitch_LUC}. The set of all weakly left uniformly continuous functions on $S$ is a left invariant closed subspace of $C_b(S)$, denoted by $WLUC(S)$. Denote by M$(S)$ the set of all means 
on $C_b(S)$ and let $\tau_M = \sigma(C_b(S), \text{M}(S))$. Since $C_b(S)^*$ is linearly  spanned by $\text{M}(S)$, we have that $f\in WLUC(S)$ if and only if the mapping $s\mapsto \ell_sf$ from $S$ into ($C_b(S),\tau_M)$ is continuous. The following theorem is also due to T. Mitchell \cite{Mitch_LUC}. An alternative proof was given in \cite{L-Z2}.

\begin{thm}\label{Thm Sr convex}
Let $S$ be a semitopological semigroup. Then
\begin{enumerate}
\item $LUC(S)$ has a left invariant mean if and only if
\begin{description}\label{affine Fc}
\item[($F_{ja}$)] every jointly continuous affine representation of $S$ on a nonempty convex compact subset of a separated locally convex topological vector space has a common fixed point for $S$;
\end{description}
\item $WLUC(S)$ has a left invariant mean if and only if
\begin{description}
\item[($F_{sa}$)] every separately continuous affine representation of $S$ on a compact convex subset of a separated locally convex topological vector space has a common fixed point for $S$.
\end{description}
\end{enumerate}
\end{thm}

\begin{remark}
We note that Theorem~\ref{Thm Sr convex}.(1) 
 is no longer true if one replaces ``affine'' with ``convex''. Here by a convex map (or a convex function) we mean a map (resp. a function) defined on a convex set $K$ that maps each convex subset of $K$ onto a convex set. For example any continuous function on an interval is a convex function. A counter example is as follows: Let $f$ and $g$ be two continuous functions on the unit interval $[0,1]$ such that $f$ commutes with $g$ under composite product but they do not have a common fixed point in $[0,1]$. We consider the free discrete commutative semigroup $SC_2$ on two generators $s_1$ and $s_2$. Then $SC_2$ is amenable. Define  on $[0,1]$ by $T_{s_1}(x)=f(x)$ and $T_{s_2}(x)=g(x)$ ($x\in [0,1]$). They generate a continuous convex representation of $SC_2$ on $[0,1]$.  This representation does not have a common fixed point in $[0,1]$.
\end{remark}

\section{fixed point properties for L-embedded sets}\label{L-embedded}

A Banach space $E$ is L-embedded if the image of $E$ under the canonical embedding into $E^{**}$, the bidual space of $E$, is an $\ell_1$ summand in $E^{**}$. The Banach space $E$ is M-embedded if $E$ is an M-ideal in its bidual $E^{**}$, that is, if $E^\bot =\{\phi\in E^{***}: \phi(x) =0 \text{ for all } x\in E\}$ is an $\ell_1$-summand in $E^{***}$ (\cite{HWW}). In fact, $E$ is M-embedded if and only if $E^{***}=E^* \oplus_1 E^\bot$ (\cite[Remark~1.13]{HWW}).  The class of L-embedded Banach spaces includes all $L^1(\Sigma,\mu)$ (the space of all absolutely integrable functions on a measure space $(\Sigma,\mu)$), preduals of von Neumann algebras and the Hardy space $H_1$. A typical example of an M-embedded Banach space is $c_0(D)$ for a discrete space $D$ (see \cite[page 2953]{L-Z2}). 
 It is also well-known that $K(H)$, the Banach space of all compact operators on a Hilbert space $H$, is M-embedded. More classical M-embedded Banach spaces may be seen in \cite[Example~1.4]{HWW}. The dual space of an M-embedded Banach space is L-embedded. 
 
The authors introduced the notion of L-embedded subsets of  Banach spaces in \cite{L-Z2}.

\begin{defn}
Let $C$ be a nonempty subset of a Banach space $E$. Denote by $\overline{C}^{\text{wk*}}$ the closure of $C$ in $E^{**}$ in the weak* topology of $E^{**}$. We call $C$ \emph{L-embedded} if there is a subspace $E_s$ of $E^{**}$ such that $E+E_s = E\oplus_1 E_s$ in $E^{**}$ and $\overline{C}^{\text{wk*}}\subset C\oplus_1 E_s$, that is, for each $u\in \overline{C}^{\text{wk*}}$ there are $c\in C$ and $\xi\in E_s$ such that $u = c+ \xi$ and $\|u\|=\|c\| + \|\xi\|$.
\end{defn}
 Trivially, every L-embedded Banach space is L-embedded as a subset of itself. From the definition, every L-embedded set $C$ in a Banach space $E$ is necessarily  weakly closed.

Every weakly compact subset $C$ of a Banach space is L-embedded since $\overline C^{\text{wk*}} = C$ in the case. A Banach space $E$ is L-embedded if and only its unit ball is L-embedded subset of $E$. But a closed convex subset of an L-embedded Banach space may not be L-embedded. For example, the set of all means on $\ell^\infty$ is a convex closed (and even weak* compact) subset of $(\ell^\infty)^*$, but it is not L-embedded  although $(\ell^\infty)^*$ (as the dual space of a von Neumann algebra) is L-embedded. For a $\sigma$-finite measure space $(\Sigma, \mu)$, it was shown in \cite[Theorem~1.1]{B-L} that a closed convex bounded set in $L^1(\Sigma, \mu)$ is L-embedded if and only if it is closed with respect to the local measure topology. A generalization of this result in the operator algebra setting was given in \cite[Theorem~3.5]{D-D-S-T}.

The following result was shown in \cite{L-Z2}.

\begin{thm}\label{discrete FL}
Let $S$ be a left reversible discrete semigroup. Then $S$ has the following fixed point property.
\begin{description}
\item[($F_L$)] Whenever $\Sc = \{T_s: s\in S\}$ is a representation of $S$ as norm nonexpansive self mappings on a nonempty L-embedded convex subset $C$ of a Banach space $E$ with each $T_s$ being weakly continuous on every weakly compact $S$-invariant convex subset of $C$, if $C$ contains a nonempty bounded subset $B$ such that $T_s(B) = B$ for all $s\in S$, then $C$ has a common fixed point for $S$.
\end{description}
\end{thm}

Recall that a semitopological semigroup $S$ is \emph{strongly left reversible} if there is a family of countable subsemigroups $\{ S_\al :\; \al \in I\}$ such that
\begin{enumerate}
\item[(1)] $S = \cup_{\al\in I}S_\al$,
\item[(2)] $\overline{aS_\al} \cap \overline{bS_\al} \ne \emptyset$ for each $\al\in I$ and $a, b\in S_\al$,
\item[(3)] for each pair $\al_1, \al_2 \in I$, there is $\al_3\in I$ such that $S_{\al_1}\cup S_{\al_2} \subset S_{\al_3}$.
\end{enumerate}
A metrizable left reversible semitopological semigroup is always strongly left reversible \cite[Lemma~5.2]{L-Z1}. Using this fact and the same proof as for Theorem~\ref{discrete FL} one can derive the following theorem.

\begin{thm}\label{metrizable FL}
Let $S$ be a metrizable left reversible semitopological semigroup. Let $\Sc = \{T_s: s\in S\}$ be a norm nonexpansive representation of $S$ on a nonempty L-embedded  convex subset $C$ of a Banach space $E$ such that $C$ contains a nonempty bounded subset $B$ with $T_s(B) = B$ for all $s\in S$. If for every weakly compact $S$-invariant convex subset $M$ of $C$  the mapping $(s, x)\mapsto T_s(x)$ is jointly continuous from $S\times M$ into $M$ when $M$ is endowed with the weak topology of $E$, then $C$ has a common fixed point for $S$.
\end{thm}

 Again, the same method combined with \cite[Theorem~3]{Lim} leads to the following theorem.

\begin{thm}\label{normal structure}
Let $S$ be a left reversible semitopological semigroup and let $\Sc =\{T_s: s\in S\}$ be a representation of $S$ as norm nonexpansive and separately continuous self mappings on a nonempty L-embedded convex subset $C$ of a Banach space $E$. Suppose that $C$ has normal structure and contains a nonempty bounded subset $B$ such that $T_s(B) = B$ for all $s\in S$. Then $C$ has a common fixed point for $S$.
\end{thm}

For affine nonexpansive mappings on a closed convex L-embedded set of a Banach space the following was obtained in \cite{L-Z2}.

\begin{thm}\label{WAP has LIM}
Let $C$ be a nonempty closed convex set in a Banach space and $S$ be a semitopological semigroup such that $WAP(S)$ has a LIM. Suppose that $\Sc=\{T_s: s\in S\}$ is a representation of $S$ as separately continuous nonexpansive affine mappings on $C$. If $C$ is L- embedded and there is a bounded set $B\subset C$ such that $T_s(B) = B$ for all $s\in S$, then $C$ has a common fixed point for $S$.
\end{thm}

If $S$ is a locally compact group then $WAP(S)$ always has a LIM and, for each $x\in C$, $B= Sx$ always satisfies $T_s(B) = B$ for all $s\in S$. So, if there is $x\in C$ such that $Sx$ is bounded, then every separately continuous nonexpansive affine representation of $S$ on a nonempty L-embedded convex set of a Banach space has a common fixed point for $S$. Thus,
Theorem~\ref{WAP has LIM} extends \cite[Theorem~A]{BGM} of  Bader, Gelander and  Monod, which was used to give a short proof to resolve the derivation problem. This long-standing  problem was first solved by V. Losert in \cite{Losert}.

For a discrete semigroup $S$ it is well-known that the left reversibility of $S$ implies that $WAP(S)$ has a left invariant mean and the converse is not true due to Theorem~\ref{bicyclic}. For general semitopological semigroups $S$,  the relation between the left reversibility of $S$ and the existence of a left invariant mean for $WAP(S)$ is still unknown. 

We recall that if $S$ acts on a Hausdorff space $X$, the action is quasi-equicontinuous if $\overline{S}^p$, the closure of $S$ in the product space $X^X$, consists entirely of continuous mappings. Let $C$ be a closed convex subset of a Banach space $E$ and let $S$ act on $C$ as self mappings. We say that the action is hereditary weakly quasi-equicontinuous in $C$ if for each weakly compact $S$-invariant convex subset $K$ of $C$ the $S$-action on $(K, \text{wk})$ is quasi-equicontinuous, where $(K, \text{wk})$ is $K$ with the weak topology of $E$. We note that, if the $S$-action on $C$ is weakly equicontinuous, then it is hereditary quasi-equicontinuous in $C$. By \cite[Lemma~3.1]{L-Z1}, if $C$ is weakly compact and the $S$-action on $C$ is weakly quasi-equicontinuous, then the $S$-action is hereditary weakly quasi-equicontinuous in $C$. The following was shown in \cite{L-Z2}.

\begin{thm}\label{FL}
Let $S$ be a separable semitopological semigroup. Suppose that $WAP(S)$ has a left invariant mean. Then $S$ has the following fixed point property.
\begin{description}
\item[($F_L'$)] If $S$ acts on a nonempty convex L-embedded subset $C$ of a Banach space as norm nonexpansive and hereditary weakly quasi-equicontinuous mappings for which the mapping $s \mapsto T_s(x)$ ($s\in S$) is weakly continuous whenever $x$ belongs to any  weakly compact $S$-invariant convex subset of $C$ and if $C$ contains a nonempty bounded subset $B$ such that $T_s(B) = B$ ($s\in S$), then there is a common fixed point for $S$ in $C$.
\end{description}
\end{thm}

The above theorem was used in \cite{L-Z2} to provide the following example.

\begin{example}
Let $G$ be a non-amenable group and let $C$ be the set of all means on $\ell^\infty(G)$. 
Then $C$ is not L-embedded although it is a weak* compact convex subset of $\ell^\infty(G)^*$ 
 which, as the dual space of the von Neumman algebra $\ell^\infty(G)$, is indeed an L-embedded Banach space. As a consequence,  the set of all means on $\ell^\infty$ is not L-embedded.
\end{example}

For a semitopological semigroup $S$, simply examining the representation of $S$ on the weak* compact convex subset of all means on $LUC(S)$ defined by the dual of left translations on $LUC(S)$, one sees that if the following fixed point property holds then $LUC(S)$ has a left invariant mean.

\begin{description}
\item[($F_*$)]\label{F*} Whenever $\Sc = \{T_s:\; s\in S\}$ is a representation of $S$ as norm non-expansive mappings on a non-empty weak* compact convex set $C$ of the dual space of a Banach space $E$ and the mapping $(s,x)\mapsto T_s(x)$ from $S\times C$ to $C$ is jointly continuous, where $C$ is equipped with the weak* topology of $E^*$, then there is a common fixed point for $S$ in $C$.
\end{description}

 Whether the converse is true is an open problem raised in 
\cite{Lau.survey}.
For a discrete semigroup acting on a subset of the dual of an M-embedded Banach space, the following was given in \cite{L-Z2}.

\begin{thm}\label{F*M}
Let $S$ be a discrete left reversible semigroup. Then $S$ has the following fixed point property.
\begin{description}
\item[($F_{*M}$)]\label{FM} Whenever $\Sc = \{T_s:\; s\in S\}$ is a representation of $S$ as weak* continuous, norm nonexpansive mappings on a nonempty weak* compact convex set $C$ of the dual space $E^*$ of an M-embedded Banach space $E$, $C$ contains a common fixed point for $S$.
\end{description}
\end{thm}

For non-discrete $S$ we have:

\begin{thm}\label{F*Mj}
Let $S$ be a left reversible semitopological semigroup. Then $S$ has the following fixed point property.
\begin{description}
\item[($F_{*Mj}$)]\label{FMj} If $\Sc = \{T_s:\; s\in S\}$ is a norm nonexpansive representation of $S$ on a nonempty weak* compact convex set $C$ of the dual space $E^*$ of an M-embedded Banach space $E$ and the mapping $(s,x) \mapsto T_s(x)$ from $S\times C$ into $C$ is jointly continuous when $C$ is endowed with the weak* topology of $E^*$, then $C$ contains a common fixed point for $S$.
\end{description}
\end{thm}

\begin{remark} 
Since the Banach space $K(H)$ of all compact operators on a Hilbert space $H$ is M-embedded, it follows that both Theorems \ref{F*M} and \ref{F*Mj} apply to $K(H)^* = J(H)$, the space of trace operators on $H$.
\end{remark}

\section{Fixed point properties for F-algebras}

In \cite{K3} (see also \cite{K1,K2,Lau_invar, L-P-W,Lau-Wong}), Ky Fan established the following remarkable ``Invariant Subspace Theorem" for left amenable semigroups:
\begin{thm}
Let $S$ be a left amenable semigroup, and let $\Sc = \{T_s: s\in S\}$ be a representation of $S$ as continuous linear operators on a separated locally convex space $E$. Then the following property holds:
\begin{description}
\item[(KF)]\label{K} If $X$ is a subset of $E$ (containing an $n$-dimensional subspace) such that $T_s(L)$ is an $n$-dimensional subspace contained in $X$ whenever $L$ is one and $s\in S$, and there exists a closed $\Sc$-invariant subspace $H$ in $E$ of codimension $n$ with the property that $(x+H)\cap X$ is compact convex for each $x\in E$, then there exists an $n$-dimensional subspace $L_0$ contained in $X$ such that $T_s(L_0) = L_0$ for all $s\in S$.
\end{description}
\end{thm}
The origin of Ky Fan's Theorem lies in the earlier investigations of Pontrjagin, Iovihdov, Krein and Naimark concerning invariant subspaces for Lorentz transformations on a Hilbert space \cite{Iohv,I-K, Krein, Naimark1, Naimark2, Pontr}. In physics, a Lorentz transformation is an invertible linear mapping on $\Rb^4$ that describes how a measurement of space and time observed in a frame of reference is converted into another frame of reference. From special relativity Lorentz transformations may be characterized as invertible linear mappings that preserve the quadratic form
\[ 
J(\vec x) = x^2 + y^2 + z^2 - c^2 t^2, \quad \vec x = (x, y, z, t)\in \Rb^4
\]
where the constant $c$ is the speed of light.
Quantity $J$ represents the space time interval. 
 It is a well known fact that for any Lorentz transformation $T$ there is a three dimensional subspace $V$ of $\Rb^4$ which is $T$-invariant and  positive (in the sense that $T(\vec x) = \vec x$ and $J(\vec x)\geq 0$ for all $\vec x \in V$). 
 
L. S. Pontrjagin \cite{Pontr}, I. S. Iovihdov \cite{Iohv}, M. G. Krein \cite{Krein, I-K} and M. A. Naimark \cite{Naimark1,Naimark2} investigated infinite-dimension versions of the above invariant subspace property, and Naimark finally established the following theorem in 1963 \cite{Naimark1}.
\begin{thm}
Let $n>0$ be an integer. Consider the quadratic form on $\ell^2$ given by
\[
J_n(x) = \sum_{i=1}^n|x_i|^2 - \sum_{i=n+1}^\infty |x_i|^2, \quad x=(x_1,x_2,\cdots)\in \ell^2.
\]
Suppose that $G$ is a commutative group of continuous,  invertible, $J_n$-preserving, linear transformations on $\ell^2$. Then there is a $G$-invariant $n$-dimensional subspace $V$ of $\ell^2$ which is positive (in the sense that $J_n(x) \geq 0$ for all $x\in V$).
\end{thm}
 
To understand the conditions on the subset $X$ in Ky Fan's property (KF) we note that $X=\{x\in \ell^2: J_n(x)\geq0\}$ in the setting of the above theorem satisfies indeed the requirement in (KF). However, there is no mention of an invariant subspace $H$ in Naimark's result. Ky Fan's Theorem removes the commutative group and Hilbert space restrictions of Naimark's Theorem, replacing them by amenability and the conditions involving $H$.

Let $A$ be a Banach algebra and let $X$ be a Banach left, right or two-sided $A$-module. Then $X^*$ is respectively a Banach right, left or two-sided $A$-module with the corresponding module action(s) defined naturally by 
\[
\la x, f\cdot a\ra = \la a\cdot x, f\ra, \;\la x, a\cdot f\ra = \la x\cdot a, f\ra \quad (a\in A, f\in X^*, x\in X).
\]
Let $X$ be a Banach $A$-bimodule. A linear mapping $D$: $A\to X$ is called a \emph{derivation} if it satisfies
\[
 D(ab) = a\cdot D(b) + D(a)\cdot b \quad (a,b\in X).
 \]
 Derivations in the form $D(a) = a\cdot x_0 - x_0\cdot a$ ($a\in A$) for some fixed $x_0\in X$ are called inner derivations.

A Banach algebra $A$ is an \emph{F-algebra} \cite{Lau_F} (also known as Lau algebras \cite{Pier}) if it is the (unique) predual of a $W^*$-algebra $ \Mf$ and the identity $\e$ of $ \Mf$ is a multiplicative linear functional on $A$.  Since $A^{**} =  \Mf^*$, we denote by $P_1(A^{**})$, the set of all normalized positive linear functionals on $ \Mf$, that is 
\[
P_1(A^{**}) =\{m\in A^{**}: m\geq 0, m(\e) =1\}.
\]
In this case $P_1(A^{**})$ is a semigroup with the (two) Arens multiplications.

 Examples of F-algebras include the predual algebras of a Hopf von Neumann algebra (in particular, quantum group algebras), the group algebra $L^1(G)$ of a locally compact group $G$, the Fourier algebra $A(G)$ and the Fourier-Stieltjes algebra $B(G)$ of a topological group $G$ (see \cite{D-L-S, Lau_F, Lau-Ludwig}). They also include the measure algebra $M(S)$ of a locally compact semigroup $S$. Moreover, the hypergroup algebra $L^1(H)$ and the measure algebra $M(H)$ of a locally compact hypergroup $H$ with a left Haar measure are F-algebras. In this case, it was shown in \cite[Theorem~5.2.2]{Willson} (see also \cite[Remark~5.3]{Willson2}) that $(L^1(H))^*=L^\infty(H)$ is not a Hopf von Neumann algebra unless $H$ is a locally compact group.

Let $G$ be a locally compact group with a fixed left Haar measure $\lambda$. Then the group algebra $L^1(G)$ is the Banach space of $\lambda$-integrable functions with product
\[
f*g(x) = \int_G{f(y)g(y^{-1}x)}d\lambda(y) \quad (x\in G).
\]
If $S$ is a locally compact semigroup, then the measure algebra $M(S)$ is the space of regular Borel measures on $S$ with the total variation norm and the convolution product defined by 
\[
\la\mu*\nu, f\ra = \iint_{G\times G}{f(xy)}d\nu(y)d\mu(x) \quad (f\in C_0(S)),
\]
where $C_0(S)$ denotes the space of continuous functions vanishing at $\infty$.
If $S$ is discrete, then $M(S)=\ell^1(S)$.

In his seminal paper \cite{Eymard} P. Eymard (see also \cite{K-L}) associated to a locally compact group $G$ two important commutative Banach algebras. These are the Fourier algebra $A(G)$ and the Fourier-Stieltjes algebra $B(G)$. The latter is indeed the linear span of the set of all continuous positive definite complex-valued functions on $G$. This is also the space of the coefficient functions of the unitary representations of the group $G$. More precisely, given $u\in B(G)$ there exists a unitary representation $\pi$ of $G$ and two vectors $\xi$ and $\eta$ in the representation Hilbert space $H(\pi)$ of $\pi$ such that 
\[
u(x) = \la\pi(x)\xi, \eta\ra \quad (x\in G).
\]
Equipped with the norm $\|u\| = \inf_{\xi,\eta}\|\xi\|\,\|\eta\|$ and the pointwise multiplication $B(G)$ is a commutative Banach algebra, where the infimum is taking on all $\xi$ and $\eta$ satisfying the preceding equality. As a Banach algebra $B(G)$ is also the dual space of the group C*-algebra $C^*(G)$. The Fourier algebra $A(G)$ is the closed ideal of $B(G)$ generated by the elements of $B(G)$ with compact supports. The algebra $A(G)$ can also be defined as the set of coordinate functions of the left regular representations of $G$ on $L^2(G)$. When $G$ is abelian, via Fourier transform we have
\[
A(G) = L^1(\hat G), \quad B(G) = M(\hat G) \quad \text{and } C^*(G) = C_0(\hat G),
\]
where $\hat G$ is the dual group of $G$.

We call a semitopological semigroup $S$  \emph{extremely left amenable} if there is a left invariant mean $m$ on $LUC(S)$ which is multiplicative, that is it satisfies further
 \[
 m(fg) = m(f)m(g)\quad (f,g \in LUC(S)).
 \]

Let $A$ be an F-algebra. Elements of $P_1(A^{**})$ are called \emph{means} on $A^* =\Mf$. It is well-known that 
\[
P_1(A) = P_1(A^{**})\cap A
\]
  is a topological semigroup with the product and topology carried from $A$ (this may be regarded as a consequence of \cite[Proposition~1.5.2]{Sakai}), and $P_1(A)$ spans $A$ (see \cite[Theorem~1.14.3]{Sakai}). A mean $m \in P_1(A^{**})$ on $A^*$ is called a \emph{topological left invariant mean}, abbreviated as TLIM,  if $a\cdot m = m$ for all $a\in P_1(A)$, in other words $m \in P_1(A^{**})$ is a TLIM if $m(x\cdot a) = m(x)$ for all $a\in P_1(A)$ and $x\in A^*$. 

 An F-algebra $A$ is called \emph{left amenable} if, for each Banach $A$-bimodule $X$ with the left module action specified by $a\cdot x = \langle a, \e\rangle x$ ($a\in A$, $x\in X$), every continuous derivation from $A$ into $X^*$ is inner. The following was shown in \cite{Lau_F} (see Theorems~4.1 and 4.6 there).

\begin{prop}\label{left amen1-3}
Let $A$ be an F-algebra. Then the following are equivalent.
\begin{enumerate}
\item There is a TLIM for $A^*$. \label{TLIM}
\item  
The algebra $A$ is left amenable.\label{inner}
\item There exists a net $(m_\al)\subset P_1(A)$ such that $am_\al - m_\al \to 0$
in norm topology for each $a\in P_1(A)$.\label{appTLIM}
\end{enumerate}
\end{prop}
 
We note here that, being F-algebras, the group algebra $L^1(G)$, the measure algebra $M(G)$ of a locally compact group $G$ are left amenable if and only if $G$ is an amenable group; while the Fourier algebra $A(G)$ and the Fourier-Stieltjes algebra $B(G)$ are always left amenable \cite{Lau_F}. The hypergroup algebra $L^1(H)$ of a locally compact hypergroup $H$ with a left Haar measure is left amenable if and only if $H$ is an amenable hypergroup \cite{Skan}. Also the left amenability of the predual algebra of a Hopf von Neumann algebra, as an F-algebra, coincides with that studied in \cite{Ruan, Voic} (see also \cite{B-T} and references therein).

A \emph{metric semigroup} is a semitopological semigroup whose topology is generated by a metric $d$. We consider the following fixed point property for a metric semigroup $S$.

\begin{description}
\item[($F_U$)] If $ \Sc = \{T_s: s\in S\}$ is a separately continuous representation of $S$ on a compact subset $K$ of a locally convex space $(E,Q)$ and if the mapping $s\mapsto T_s(y)$ from $S$ into $K$ is uniformly continuous for each $y\in K$, then $K$ has a common fixed point for $S$.\label{FU}
\end{description}

Note that the mapping $s \mapsto T_s(y)$ is \emph{uniformly continuous} if for each $\tau\in Q$ and each $\epsilon >0$ there is $\delta>0$ such that 
\[
\tau(T_s(y) - T_t(y)) \leq \epsilon
\]
 whenever $d(s,t) \leq \delta$. For example, suppose that $S$ is a subset of a locally convex space $L$ that acts on $E$ such that $(a,y)\mapsto ay$: $L\times E \to E$ is separately continuous; if $a\mapsto ay$ is linear in $a\in L$ for each $y\in E$, then the induced action of $S$ on $E$, $(s,y)\mapsto sy$ ($s\in S$), is uniformly continuous in $s$ for each $y\in E$.

Let $A$ be an F-algebra. As we have known, $P_1(A)$ is indeed a metric topological semigroup with the product and topology inherited from $A$. The following was proved in \cite{L-Z3}.

\begin{prop}\label{fixed pt}
The F-algebra $A$ is left amenable if and only if the metric semigroup $P_1(A)$ has the fixed point property ($F_U$).
\end{prop}

A semitopological semigroup $S$ is \emph{extremely left amenable} if $LUC(S)$ has a multiplicative left invariant mean. Granirer showed in \cite{Gran_ELA} that a discrete semigroup is extremely left amenable if and only if any two elements of it have a common right zero (see \cite[Theorem~4.2]{L-Z2} for a short proof). It is due to  Mitchell \cite{Mitch_LUC} that a semitopological semigroup $S$ is extremely left amenable if and only if it has the following fixed point property.
\begin{description}
\item[($F_E$)] Every jointly continuous representation of $S$ on a compact Hausdorff space $C$ has a common fixed point in $C$.
\end{description}

 For an F-algebra $A$ it is pleasing that the left amenability of $A$ is equivalent to the extreme left amenability of $P_1(A)$ as the authors revealed in \cite{L-Z3}. 

\begin{thm}\label{ELA}
Let $A$ be an F-algebra. Then $A$ is left amenable if and only if $P_1(A)$ has the fixed point property ($F_E$).
\end{thm}

According to Theorem~\ref{ELA}, we  can characterize amenability of a group (resp.  semigroup) in terms of the fixed point property of normalized positive functions in the group/semigroup algebra.

\begin{cor}\label{amen gp semigp}
Let $G$ be a locally compact group and let $S$ be a semigroup. Then
\begin{enumerate}
\item The group $G$ is amenable if and only if the metric semigroup $P_1(G)=\{f\in L^1(G), f\geq 0, \|f\|_1 = 1\}$ has the fixed point property ($F_E$).
\item The semigroup $S$ is left amenable if and only if the metric semigroup $P_1(S)=\{f\in \ell^1(S), f\geq 0, \|f\|_1 = 1\}$ has the fixed point property ($F_E$).
\end{enumerate}
\end{cor}

\begin{remark}
For $A=A(G)$, the Fourier algebra of a locally compact group $G$, the necessity part of Theorem~\ref{ELA} was obtained earlier in \cite{Lau_AG}.
\end{remark}

\begin{remark}
From Theorem~\ref{ELA}, a locally compact group is amenable if and only if $LUC(S_G)$ has a multiplicative left invariant mean, where $S_G$ is the metric semigroup $P_1(G) =\{f\in L^1(G): f\geq 0, \la f, 1 \ra =1\}$. As a consequence, $AP(S_G)$ has a left invariant mean (which is equivalent to the left reversibility of $\overline{(S_G)^a}$, the almost periodic compactification of $S_G$) if  $G$ is amenable. 
\end{remark}

\begin{remark}
A common fixed point property for affine actions of $P_1(A)$ with a weak topology on compact convex sets has been studied in \cite{D-N-N} for left amenable F-algebras $A$.
\end{remark}

Let $E$ be a separated locally convex vector space and $X$ a subset of $E$. Given an integer $n>0$ we denote  by $\Lc_n(X)$ the collection of all $n$-dimensional subspaces of $E$ that are included in $X$. Let $S$ be a semigroup and $\Sc =\{T_s: s\in S\}$ a linear representation of $S$ on $E$. We say that $X$ is \emph{$n$-consistent} with respect to $S$ if $\Lc_n(X)\neq \emptyset$ and $\Lc_n(X)$ is $S$-invariant, that is $T_s(L) \in \Lc_n(X)$ for all $s\in S$ whenever $L\in \Lc_n(X)$. We say that the representation $\Sc$ is \emph{jointly continuous on compact sets} if the following is true: For each compact set $K\subset E$, if $(s_\al) \subset S$ and $(x_\al)\subset K$ are such that $s_\al \overset{\al}{\to}s\in S$,  $x_\al \overset{\al}{\to}x\in K$ and $T_{s_\al}(x_\al)\in K$ for all $\al$, then $T_{s_\al}(x_\al)\overset{\al}{\to}T_s(x)$. Obviously, if the mapping $(s,x) \mapsto T_s(x)$: $S\times E \to E$ is continuous, then $\Sc$ is jointly continuous on compact sets.
 The following result was proved in \cite{L-Z3}.

\begin{thm}\label{KF thm}
Let $A$ be an F-algebra. If $A$ is left amenable then $S=P_1(A)$ has the following $n$-dimensional invariant subspace property for each $n>0$.
\begin{description}
\item[($F_n$)] Let $E$ be a separated locally convex vector space and $\Sc$ a linear representation of $S=P_1(A)$ on $E$ such that the mapping $s\mapsto T_s(x)$ is continuous for each fixed $x\in E$ and $\Sc$ is jointly continuous on compact subsets of $E$. If $X$ is a subset of $E$ $n$-consistent with respect to $S$, and if there is a closed $S$-invariant subspace $H$ of $E$ with codimension $n$ such that $(x+H)\cap X$ is compact for each $x\in E$, then there is $L_0\in \Lc_n(X)$ such that $T_s(L_0) = L_0$ ($s\in S$). \end{description}
Conversely, if ($F_n$) holds for some $n>0$, then $A$ is left amenable.
\end{thm}
 As a consequence of Theorems~\ref{KF thm} and \ref{ELA} we can get the following result that settles the open problem 5 in \cite{Lau_finite}.

\begin{cor}\label{quantum}
Let $\Gb$ be a locally compact quantum group. Then $L^1(\Gb)$ is left amenable if and only if the topological semigroup $P_1(L^1(\Gb))$ has the fixed point property ($F_E$) if and only if $P_1(L^1(\Gb))$ has the \mbox{finite} invariant subspace property ($F_n$) for some integer (and then for all integers) $n\geq 1$.
\end{cor}

\section{Nonlinear actions on unbounded sets}

In \cite{A-T}  Atsushiba and Takahashi introduced the concept of common attractive points for a nonexpansive representation of a semigroup $S$ on a set $C$ in a Hilbert space $H$. They showed that $F(S)\neq \emptyset$ for commutative $S$ if there is a common attractive point for $S$ \cite[Lemma~3.1]{A-T}. They showed further that for commutative semigroups $S$, if $\{T_sc, s\in S\}$ is bounded for some $c\in C\subset H$, then the set $A_C(\Sc)$ of all attractive points of $\Sc$ is not empty. As a consequence, $F(S)\neq \emptyset$ \cite[Theorem~4.1]{A-T}. We note that the assumption that $\{T_sc, s\in S\}$ is bounded for some $c\in C\subset H$ cannot be dropped. Indeed, by a classical result of  W. Ray in \cite{Ray}, for every unbounded convex subset $C$ of a Hilbert space there is a nonexpansive mapping $T_0$: $C\to C$ that has no fixed point in $C$. In fact,  $\{T_0^n(c): n\in \Nbb\}$  is unbounded for all $c\in C$. Hence the representation $\mathcal N = \{T_0^n: n\in \Nbb\}$ of $(\Nbb, +)$ does not have a common fixed point in $C$. An investigation continuing that of \cite{A-T} may be seen in \cite{T-W-Y}.
 
 Let $\Sc = \{T_s:\; s\in S\}$ be a representation of a semigroup $S$ on a convex subset $C$ of a Banach space $E$.  A point $a\in E$ is an \emph{attractive point} of $\Sc$ 
if $\|a-T_sx\| \leq \|a-x\|$ for all $x\in C$. The set of all attractive points of $\Sc$ for $C$ is denoted by $A_C(\Sc)$.

Recall that a Banach space $E$ is \emph{strictly convex} if $\|\frac{x+y}{2}\| < 1$ whenever $x,y\in E$, $\|x\| = \|y\| =1$ and $x\neq y$. It is readily seen that for any distinct elements $x,y_1, y_2$ from a strictly convex space with $\|x-y_1\| = \|x-y_2\| = d$ we have $\|x-\frac{y_1+y_2}{2}\| < d$. $E$ is \emph{uniformly convex} if for each $0<\ep\leq 2$ there exists $\delta >0$ such that $\|\frac{x+y}{2}\|< 1-\delta$ whenever $x,y\in E$, $\|x\| = \|y\| = 1$ and $\|x-y\|\geq \ep$. It is known that if $E$ is uniformly convex then it is strictly convex and reflexive. Typical examples of a uniformly convex space are $L^p$-spaces ($p>1$).

Now suppose $E$ is a strictly convex and reflexive Banach space.
Let $C\neq \emptyset$ be a convex subset of $E$. For any $x\in E$,  it was shown in \cite{L-Z4} that there is a unique $u\in \overline C$, the norm closure of $C$ such that $\|u-x\|\leq \|c-x\|$ for all $c\in C$.  If $C$ is closed, then
 we  call  this $u$ the \emph{metric projection} of $x$ in $ C$ and denote it by $P_C(x)$. 
 If $E$ is a Hilbert space then $P_C(x)$ may also be characterized as the unique element $u\in C$ satisfying
\begin{equation}\label{mp}
\text{Re}\la x-u \, | \, u-c \ra \geq 0 \quad c\in C.
\end{equation}

\begin{lemma}\label{As to Fs}
Suppose that $E$ is a strictly convex and reflexive Banach space. Let $C\neq \emptyset$ be a closed convex subset of $E$ and $\Sc$ be a representation of a 
semigroup $S$ on $C$. If $a\in E$ is an attractive pint of $\Sc$, then $P_C(a)$ is a common fixed point for $S$ in $C$.
\end{lemma}

\begin{proof}
Let $u = P_C(a)$. 
Since $a$ is attractive,
\[
\|a - T_tu\|\leq \|a-u\| \leq \|a-c\| 
\]
for all $c\in C$ ($t\in S$). Thus $T_tu = P_C(a) = u$ for all $t\in S$.
\end{proof}

\begin{remark}
  If $E$ is a general Banach space, the proof of Lemma~\ref{As to Fs} still works as long as $P_C (a)$ is uniquely defined.
\end{remark}

The converse of Lemma~\ref{As to Fs} is not true in general. Namely, even if $F(S)\neq \emptyset$, it still can happen that $A_C(\Sc) = \emptyset$. For example, let $E = \ell^p$ ($p\geq 1$). Then  $C = \{x\in E:\, x = (x_i)_{i=1}^\infty, x_1 \geq0\}$ is a  closed convex set in $\ell^p$. Consider $T((x_i)) = (x_1,x_1, x_2, x_3,\cdots)$. Then $T$ is a continuous mapping on $C$, and $T$ has fixed point $\hat0 = (0,0,0,\cdots)$. But T has no attractive point for $C$. As a consequence, the representation $\mathcal N = \{T^n:\, n\in \Nbb\}$ of $(\Nbb, +)$ has a common fixed point in $C$ but  has no attractive points for $C$.

However, if the representation is non-expansive, then any $x\in F(S)$ is clearly an attractive point of $\Sc$.
Therefore, we can conclude the following.

\begin{prop}\label{equiv 1}
Let $E$ be a reflexive, strictly convex Banach space and $C\neq \emptyset$ a closed convex subset of $E$. Suppose that $\Sc$ is a representation of a semigroup $S$ on $C$ as nonexpansive self mappings. Then $A_{C}(\Sc)\neq \emptyset$ if and only if $F(S) \neq \emptyset$. Moreover, $F(S) \subset A_C(\Sc)$.
\end{prop}

It was shown in \cite{L-Z4} that $A_C(\Sc)\neq \emptyset$ for a nonexpansive representation $\Sc$ of $S$ on a subset $C$ of a Hilbert space if $S$ has a certain amenability property. In particular, we have the following theorem.

\begin{thm}\label{nonempty As}
Let $C\neq \emptyset$ be a subset of a Hilbert space $H$ and $\Sc$ be a representation of a semitopological semigroup $S$ on $C$ as nonexpansive self mappings. Suppose that $\{T_sc: s\in S\}$ is bounded for some $c\in C$. Then $A_C(\Sc) \neq \emptyset$ if any of the following conditions holds.
\begin{enumerate}
\item $C_b(S)$ has a LIM and the mapping $s\mapsto T_sc$ is continuous from $S$ into $(C,\text{wk})$;
\item $S$ is left amenable and the action of $S$ on $C$ is weakly jointly continuous;
\item $AP(S)$ has a LIM and the action of $S$ on $C$ is weakly separately continuous and weakly equicontinuous continuous;
\item $WAP(S)$ has a LIM mean and the action of $S$ on $C$ is weakly separately continuous and weakly quasi-equicontinuous.
\end{enumerate}
If, in addition, $C$ is convex and closed, then $F(S) \neq \emptyset$. 
\end{thm}

For left (right) reversible semigroups we obtained the following in \cite{L-Z4}.

\begin{thm}\label{left rev}
Let $S$ be a left reversible and separable semitopological semigroup, and let $\Sc=\{T_s: s\in S\}$ be a representation of $S$ on a weakly closed subset $C\neq \emptyset$ of a Hilbert space $H$ as norm nonexpansive and weakly jointly continuous self mappings. If there is $c\in C$ such that $\{T_s c: s\in S\}$ is bounded, then $A_C(\Sc) \neq \emptyset$. In particular $F(S) \neq \emptyset$ if $C$ is convex.
\end{thm}


Let $E$ be a Banach space and $C\subset E$. We call a mapping $T$: $C\to C$ a \emph{generalized hybrid mapping} \cite{K-T-Y} if there are numbers $\al, \beta \in \Rb$ such that
\[
\al\|Tx - Ty\|^2 + (1-\al)\|x - Ty\|^2 \leq \beta \|Tx - y\|^2 + (1 - \beta)\|x - y\|^2.
\]
for all $x,y \in C$.

When $(\al,\beta) = (1,0)$, this indeed defines a nonexpansive mapping. However,  the composite of two generalized hybrid mappings is usually no longer a generalized hybrid mapping, and a generalized hybrid mapping may be discontinuous. For a semigroup action generated by this type of mappings, the following results were proved in \cite{L-Z4}.

\begin{thm}\label{hybrid As}
Let $C\neq \emptyset$ be a  subset of a Hilbert space $H$ and $\Sc$ be a representation of a semitopological semigroup $S$ on $C$.  Suppose that $S$ is generated by a subset $\Lambda$ and, for each $s\in \Lambda$, $T_s$ is a generalized hybrid mapping on $C$, and suppose that $\{T_sc: s\in S\}$ is bounded for some $c\in C$. Then $A_C(\Sc) \neq \emptyset$ 
if any of the following conditions holds.
\begin{enumerate}
\item $C_b(S)$ has a left invariant mean and the mapping $s\mapsto T_sc$ is continuous from $S$ into $(C,\text{wk})$;
\item $S$ is left amenable and the action of $S$ on $C$ is weakly jointly continuous;
\item $AP(S)$ has a left invariant mean and the action of $S$ on $C$ is weakly separately continuous and weakly equicontinuous continuous;
\item $WAP(S)$ has a left invariant mean and the action of $S$ on $C$ is weakly separately continuous and weakly quasi-equicontinuous.
\end{enumerate}
If, in addition, $C$ is convex and closed, then $F(S) \neq \emptyset$. 
\end{thm}


\section{subinvariant submeans}

The notion of submean was first studied by Mizoguchi and Takahashi in \cite{M-T}. Further investigations and applications can be seen in \cite{L-T 96, L-T 03}.

Given a set $S$, a nonempty subset $X$ of $\ell^\infty(S)$ is called \emph{positively semilinear} if $f, g \in X$ implies $\al f + \be g \in X$ for all $\al, \be \in [0,\infty)$.

Let $X$ be a positively semilinear subset of $\ell^\infty(S)$ containing positive  constants. A function $\mu$: $X \to \Rb$ is called a \emph{submean} on $X$ if it satisfies the following conditions.
\begin{enumerate}
\item If $f, g_1, g_2 \in X$ and $\al, \be \in [0, 1]$ such that $f \leq \al g_1 + \be g_2$, then 
\[
 \mu(f) \leq \al \mu(g_1) + \be \mu(g_2),
\]
\item For every constant  $c> 0$, $\mu(c) = c$.
\end{enumerate}

It is easily seen that a submean is always continuous when $X$ is equipped with the sup norm topology of $\ell^\infty(S)$.
A submean $\mu$ is also increasing, i.e. $\mu(f) \geq \mu(g)$ if $f,g\in X$ and $f\geq g$. We call the submean $\mu$ \emph{strictly increasing} if for each constant $c>0$ there is $\de(c) >0$ such that 
\[
\mu(f + c) \geq \mu(f) + \de(c)
\]
for all $f\in X$.

Now suppose further that $S$ is a  semigroup. 
A submean $\mu$ on a left invariant, positively semilinear subset $X$ of $\ell^\infty(S)$ containing positive constants is called \emph{left subinvariant} if 
\[ \mu(l_s f) \geq \mu(f)\quad (s\in S, f\in X).
\]
 If the equality $\mu(l_s f) = \mu(f)$ holds for all $s\in S$ and $f\in X$, then we call $\mu$ \emph{left invariant}.

Trivially, if $X$ is a left invariant subspace of $\ell^\infty(S)$ containing constants, then any left invariant mean on $X$ is a strictly increasing left invariant submean on $X$. Some nonlinear examples are given as follows. 

\begin{exa}\label{sup group}
Let $S=G$ be a group. Then 
$$\mu(f) = \sup_{g\in G}f(g)\quad (f\in \ell^\infty(G))$$
 is a strictly increasing left invariant submean on $\ell^\infty(G)$.
\end{exa}

\begin{exa}\label{sup S0}
 If there is a nonempty $S_0 \subset S$ such that $sS_0 \supset S_0$ for each $s\in S$, then
  $$\mu_0(f) = \sup_{s\in S_0}f(s)\quad (f\in \ell^\infty(S))$$
  defines a strictly increasing left subinvariant submean on $\ell^\infty(S)$. In particular, if $S$ has a right zero $s_0$ so that $ss_0 = s_0$ for all $s\in S$, then $\mu_0(f) = f(s_0)$ is a strictly increasing left invariant submean on $\ell^\infty(S)$. 
  More generally, if $S$ has a left ideal $S_0=G_0$ which is a group, then $\mu_0$ defined above is a strictly increasing left invariant submean on $\ell^\infty(S)$.
\end{exa}

\begin{exa}\label{submean left reversible}
Let $S$ be a left reversible semitopological semigroup and let $\Gamma$ be the collection of all closed right ideals of $S$. Given any submean $\nu$ on a left invariant, positively semilinear subset $X$ of 
$LUC(S)$ that contains positive constants, we define
\[
\mu(f) = \inf_{J\in \Gamma}\sup_{s\in J} \nu(l_sf) \quad (f\in X).
\]
Then $\mu$ is a left subinvariant submean on $X$. If $\nu$ is strictly increasing, then so is $\mu$. 

As a special case, for discrete $S$ we can take the submean $\nu$ on $\ell^\infty(S)$ defined by $\nu(f) = \sup_{s\in S} f(s)$
. Then $\sup_{s\in J} \nu(l_sf) = \sup_{s\in J} f(s)$, and so
\[
\mu(f) = \inf_{J\in \Gamma}\sup_{s\in J} f(s) \quad (f\in \ell^\infty(S))
\]
defines a strictly increasing left subinvariant submean on $\ell^\infty(S)$.
\end{exa}

A submean $\mu$ on $\ell^\infty(S)$ is called \emph{supremum admissible} if for any bounded family $\{f_\al: \al\in \Delta\}\subset \ell^\infty(S)$
\[
\mu(\sup_{\al\in \Delta}f_\al) = \sup_{\al\in \Delta}\mu(f_\al),
\]
where $\sup_{\al\in \Delta}f_\al\in\ell^\infty(S)$ is defined by $(\sup_{\al\in \Delta}f_\al)(s) = \sup_{\al\in \Delta}(f_\al(s))$ ($s\in S$).

The submeans defined in Examples \ref{sup group} and \ref{sup S0} are supremum admissible. If $S$ is a finite semigroup, then every submean on $\ell^\infty(S)$ is supremum admissible. 

The following was proved in \cite{L-Z5}.

\begin{thm}\label{supremum admissible fpp}
Let $S$ be a 
semigroup that acts on a weak* closed convex set $K \neq \emptyset$ of a dual Banach space $E=(E_*)^*$ as norm  nonexpansive self mappings. 
Suppose that $\ell^\infty(S)$ has a strictly increasing supremum admissible left subinvariant submean. If $K$ has 
normal structure and there is $c\in K$ such that $\{T_s c: s\in S\}$ is bounded, then $K$ has a common fixed point for $S$.
\end{thm}

For $S = G$ being a group, since $\ell^\infty(G)$ always has a strictly increasing left invariant submean (Example~\ref{sup group}) that is supremum admissible we immediately get the following result.

\begin{cor}\label{G fixed point}
Let $G$ be a 
group that acts on a weak* closed convex set $K \neq \emptyset$ of a dual Banach space $E=(E_*)^*$ as norm nonexpansive self mappings. 
If $K$ has 
normal structure and there is $c\in K$ such that $\{T_g c: g\in G\}$  is bounded, then $K$ has a common fixed point for $G$.
\end{cor}

Regarding actions on a weak* compact convex subset of a dual Banach space, the following were obtained in \cite{L-Z5}. 

\begin{thm}\label{left reversible}

Let $S$ be a semitopological semigroup and $\Sc=\{T_s: s\in S\}$ be a 
norm nonexpansive representation of $S$ on a nonempty weak* compact convex subset $K$ of a dual Banach space $E=(E_*)^*$. Suppose that $K$ has the normal structure. Then $K$ contains a common fixed point for $S$ if one of the following conditions holds.
\begin{enumerate}
\item $S$ is left reversible and the representation is weak* continuous;
\item $AP(S)$ has a LIM and the representation is separately continuous and equicontinuous when $K$ is equipped with the weak* topology of $E$;
\item $LUC(S)$ has a LIM and the mapping $(s,x)\mapsto T_sx$ from $S\times K$ into $K$ is jointly continuous when $K$ is equipped with the weak* topology of $E$;
\item $RUC(S)$ has a LIM and the representation is separately continuous and equicontinuous when $K$ is equipped with the weak* topology of $E$.
\end{enumerate}

\end{thm}

\section{Some open problems}

We conclude this paper with several open questions as follows.

\begin{problem}\label{Q F*}
Does a semitopological semigroup $S$ have the fpp ($F_*$) if $LUC(S)$ has a LIM?
\end{problem}
The question is open even for discrete case \cite{Lau.survey}. It was shown in \cite[Proposition~6.1]{L-Z2} that a weak version of property ($F_*$) holds if $LUC(S)$ has a LIM.

Related to Problem~\ref{Q F*}, the following extension of \cite[Theorem~5.3]{L-T1} were proved in \cite{L-Z2}.

\begin{thm}
If $S$ is a left reversible or left amenable semitopological semigroup, then the following fixed point property holds.
\begin{description}
\item[($F_{*s}$)]  Whenever $\Sc = \{T_s:\; s\in S\}$ is a norm nonexpansive representation of $S$ on a nonempty norm separable weak* compact convex set $C$ of the dual space of a Banach space $E$ and the mapping $(s,x)\mapsto T_s(x)$ from $S\times C$ to $C$ is jointly continuous when $C$ is endowed with the weak* topology of $E^*$, then there is a common fixed point for $S$ in $C$.
\end{description}
\end{thm}

\begin{problem}\label{P_F*s}
Let $S$ be a (discrete) semigroup. If the fpp ($F_{*s}$) holds, does $WAP(S)$ have a LIM? We also do not know whether the existence of a LIM on $WAP(S)$ is sufficient to ensure the fpp ($F_{*s}$).
\end{problem}
A partial affirmative answer to Problem~\ref{P_F*s} was given in \cite[Proposition~6.5]{L-Z2}, which we quote as follows.

\begin{prop}\label{FsLIM} Suppose that $S$ has the fixed point property $(F_{*s})$. Then
\begin{enumerate}
\item $AP(S)$ has a LIM;\label{AP}
\item $WAP(S)$ has a LIM if $S$ has a countable left ideal.\label{WAP}
\end{enumerate}
\end{prop}

Consider partially bicyclic semigroups $S_2 = \la e, a, b, c \;|\; ab = e, ac = e\ra$ and $S_{1,1}=\la e, a, b, c, d \;|\; ab = e, cd = e\ra$. We know that they are not left amenable. So they do not have the fpp ($F_*$). It is worth mentioning that $WAP(S_2)$ and $AP(S_{1,1})$ both have a LIM (Theorem~\ref{bicyclic}).

\begin{problem}
Does the partially bicyclic semigroup $S_2$ have the fpp $(F_{*s})$?
\end{problem}

If the answer to the above question is yes, then, due to Theorem~\ref{bicyclic}, $S$ having $(F_{*s})$ is not equivalent to $S$ being left reversible; if the answer is no, then the converse of Proposition~\ref{FsLIM}(\ref{WAP}) does not hold even for a countable semigroup $S$.

Theorem~\ref{WAP MLIM} improves \cite[Theorem~3.8]{L-Z1} by removing the separability condition on $S$ assumed there. We wonder whether the same thing can be done to Theorem~\ref{F}. 

\begin{problem}
Does fpp (F) hold for a semitopological semigroup $S$ if $WAP(S)$ has a LIM? 
\end{problem}

If $S$ is a locally compact group, then $S$ is extremely left amenable only when $S$ is a singleton \cite{G-L70}. However, a non-trivial topological group which is not locally compact can be extremely left amenable. In fact, let $S$ be the group of unitary operators on an infinite dimensional Hilbert space with the strong operator topology, then $S$ is extremely left amenable \cite{G-M}; Theorem~\ref{ELA} shows that an F-algebra $A$ is left amenable if and only if the semigroup of normal positive functions of norm 1 on $A^*$ is extremely left amenable. For more examples we refer to \cite{Lau-Ludwig}.

\begin{problem}
Suppose that $S$ is extremely left amenable and $C$ is a weakly closed subset of a Banach space $E$, and suppose that $\Sc$ is a weakly continuous and norm nonexpansive representation of $S$ on $C$ such that $ \{T_sc: s\in S\}$ is relatively weakly compact for some $c\in C$. Does $C$ contain a fixed point for $S$? \label{Prob extreme}
\end{problem}
We know that the answer is ``yes'' when $S$ is discrete. Indeed, in this case, for each finite subset $\sigma$ of $S$ there is $s_\sigma \in S$ such that $ss_\sigma =s_\sigma$ for all $s\in \sigma$ by a theorem of Granirer's \cite{Gran_ELA} (see also \cite[Theorem~4.2]{L-Z2} for a short proof). Consider the net $\{s_\sigma c\}$. By the relative weak compactness of $Sc$, there is $z\in \overline{Sc}^\text{ wk}\subset C$ such that (go to a subnet if necessary) wk-$\lim_{\sigma}s_\sigma c = z$. Then, as readily checked, $T_sz = z$ for all $s\in S$ by the weak continuity of the $S$ action on $C$.

More generally, the answer to Problem~\ref{Prob extreme} is still affirmative (even without the norm nonexpansiveness assumption) if the representation is jointly continuous when $C$ is equipped with the weak topology of $E$. This is indeed a consequence of Theorem~\ref{Thm Sr}(a).

\begin{problem}
 Let $C$ be a nonempty closed convex subset of the sequence space $c_0$ and $\Sc$ be a representation of a commutative semigroup $S$ as nonexpansive mappings on $C$. Suppose that $\{T_sc: s\in S\}$ is relatively weakly compact for some $c\in C$. Is $F(\Sc) \neq \emptyset$?\label{Prob c0}
\end{problem}
One may not drop the relative weak compactness condition on the orbit of $c$. For example, on the unit ball of $c_0$ define $T((x_i)) = (1,x_1, x_2, \cdots)$. Then $T$ is nonexpansive, and obviously $T$ has no fixed point in the unit ball.

\begin{problem}
Let $A$ be an F-algebra. Let ($F'_n$) denote the same property as ($F_n$) with ``jointly continuous'' replaced by ``separately continuous'' on compact subsets of $E$. Does ($F'_n$) imply ($F_n$)?
\end{problem}

Regard the F-algebra $A$ as the Banach $A$-bimodule with the module multiplications given by the product of $A$. Then the dual space $A^*$ is a Banach $A$-bimodule. 
We say that a subspace $X$ of $A^*$ is topologically left (resp. right) invariant if $a\cdot X \subset X$ (resp. $X\cdot a\subset X$) for each $a\in A$; We call $X$ topologically invariant if it is both left and right topological invariant. An element of $A^*$ is almost periodic (resp. weakly almost periodic) if the map $a\mapsto f\cdot a$ from $A$ into $A^*$ is a compact (resp. weakly compact) operator. Let $AP(A)$ and $WAP(A)$ denote the collection of almost periodic and weakly almost periodic functions on $A$ respectively. Then $AP(A)$ and $WAP(A)$ are closed topologically invariant subspaces of $A^*$. Furthermore, $1\in AP(A) \subset WAP(A)$. When $G$ is a locally compact group and $A= L^1(G)$, then $AP(A)= AP(G)$ and $WAP(A)= WAP(G)$.

Let ($F_n^A$) denote the same property as ($F_n$) with joint continuity replaced by  equicontinuity on compact subsets of $E$. It is known that if $A$ satisfies ($F_n^A$) then $AP(A)$ has a TLIM 
(see \cite{Lau_AG}). 

\begin{problem}
Does the existence of TLIM on $AP(A)$ imply ($F_n^A$) for all $n\geq 1$?
\end{problem}

Let ($F_n^W$) denote the same property  ($F_n^A$) with equicontinuity on compact subsets of $E$ replaced by quasi-equicontinuity on compact subsets of $E$ (which means the closure of $S$ in the product space $E^K$, for each compact set $K\subset E$, consists only of continuous maps from $K$ to $E$). 
We have known that if $A$ satisfies ($F_n^W$) for each $n\geq 1$ then $WAP(A)$ has a TLIM (see \cite{Lau_AG}).   

\begin{problem}
Does the existence of TLIM on $WAP(A)$ imply ($F_n^W$) for all $n\geq 1$?
\end{problem}

\begin{problem}
Can separability of the semitopological semigroup in Theorem 5.5 be dropped?
\end{problem}

\end{document}